\documentclass[11pt]{amsart}
\usepackage{amssymb}
\usepackage{amsfonts}
\usepackage{amsmath}
\usepackage{xcolor}
\usepackage{bbm}

\newtheorem{theorem}{Theorem}[section]

\newtheorem{lemma}[theorem]{Lemma}
\theoremstyle{definition}

\theoremstyle{remark}

\numberwithin{equation}{section}


\begin{document}
\title[Convection-absorption Neumann problems]{Existence and location of nodal solutions for quasilinear convection-absorption Neumann problems}

\begin{abstract}
Existence of nodal (i.e., sign changing) solutions and constant sign
solutions for quasilinear elliptic equations involving convection-absorption
terms are presented. A location principle for nodal solutions is obtained by
means of constant sign solutions whose existence is also derived. The proof
is chiefly based on sub-supersolutions technique together with monotone
operator theory.
\end{abstract}

\author{Abdelkrim Moussaoui}
\address{Abdelkrim Moussaoui\\
Applied Mathematics Laboratory (LMA), Faculty of Exact Sciences,\\
and Biology department, Faculty of natural and life sciences,\\
A. Mira Bejaia University, Targa Ouzemour 06000 Bejaia, Algeria}
\email{abdelkrim.moussaoui@univ-bejaia.dz}
\author{Kamel Saoudi}
\address{Kamel Saoudi\\
Basic and Applied Scientific Research Center, Imam Abdulrahman Bin Faisal
University, P.O. Box 1982, 31441, Dammam, Saudi Arabia}
\email{kmsaoudi@iau.edu.sa }
\maketitle

\section{Introduction}

\label{S1}

Let $\Omega $ be a bounded domain in $\mathbb{R}^{N}$ ($N\geq 2)$ having a
smooth boundary $\partial \Omega .$ Given $1<p<N$, we consider the Neumann
quasilinear elliptic problem with general gradient dependence 
\begin{equation*}
\left( \mathrm{P}\right) \qquad \left\{ 
\begin{array}{ll}
-\Delta _{p}u+\frac{|\nabla u|^{p}}{u+\delta }=f(x,u,\nabla u) & \text{in}%
\;\Omega , \\ 
|\nabla u|^{p-2}\frac{\partial u}{\partial \eta }=0 & \text{on}\;\partial
\Omega ,%
\end{array}%
\right.
\end{equation*}%
where $f:\Omega \times \mathbb{R}\times \mathbb{R}^{N}\rightarrow \mathbb{R}$
is a Carath\'{e}odory function, $\delta >0$ is a small parameter, $\eta $ is
the unit outer normal to $\partial \Omega $ while $\Delta _{p}$ denotes the $%
p$-Laplace operator, namely $\Delta _{p}:=\mathrm{div}(|\nabla
u|^{p-2}\nabla u)$, $\forall \,u\in W^{1,p}(\Omega ).$

We say that $u\in W^{1,p}(\Omega )$ is a (weak) solution of $\left( \mathrm{P%
}\right) $ provided $u+\delta >0$ a.e. in $\Omega ,$ $\frac{|\nabla u|^{p}}{%
u+\delta }\in L^{1}(\Omega )$ and%
\begin{equation}
\int_{\Omega }|\nabla u|^{p-2}\nabla u\nabla \varphi \text{\thinspace }%
dx+\int_{\Omega }\frac{|\nabla u|^{p}}{u+\delta }\varphi \text{\thinspace }%
dx=\int_{\Omega }f(x,u,\nabla u)\varphi \text{\thinspace }dx,  \label{defsol}
\end{equation}%
for all $\varphi \in W^{1,p}(\Omega )\cap L^{\infty }(\Omega )$. The
requirement of $\varphi $ to be bounded is necessary since $\frac{|\nabla
u|^{p}}{u+\delta }$ is only in $L^{1}(\Omega )$.

\mathstrut

Problem $\left( \mathrm{P}\right) $ brings together a lower order term with
natural growth with respect to the gradient $\frac{|\nabla u|^{p}}{u+\delta },$ called absorption, as well as a reaction-convection term $f(x,u,\nabla u)$. Both depend on the solution and its gradient. The absorption describes a natural polynomial growth in the $\nabla u$-variable while the convection outlines a $p$-sublinear one (cf. section \ref{S2}). Note that the
absorption term dominates the diffusion operator by its growth at infinity.

Absorption and/or reaction-convection terms appear in various nonlinear
processes that occur in engineering and natural systems. In biology, they
arise in heat transfer of gas and liquid flow in plants and animals while in
geology, they are involved in thermoconvective motion of magmas and during
volcanic eruptions. They also appear in chemical processes such as in
catalytic and noncatalytic reactions, in exothermic and endothermic
reacting, as well as in global climate energy balance models \cite{K, DT}.
Moreover, convective-absorption problem $\left( \mathrm{P}\right) $ can be
associated with different class of nonlinear equations including nonlinear
Fokker-Planck equations and multidimensional formulation of generalized
viscous Burgers' equations. These equations are involved in diverse physical phenomenon such as plasma physics, astrophysics, physics of polymer fluids and particle beams, nonlinear hydrodynamics and neurophysics \cite{F}. We also mention that problems like $\left( \mathrm{P}\right) $ arise in stochastic control theory and have been first studied in \cite{LL}.
The study of Dirichlet problems involving absorption term has raised
considerable interest in recent years and has been the subject of
substantial number of papers that it is impossible to quote all of them. A
significant part are carried on semilinear problems with quadratic growth
(i.e., $p=2$). Among them, we quote \cite{ACM, OP, CMAS, GP2, CLLMMA} and
the references therein. For quasilinear Dirichlet problems we refer, for
instance, to \cite{ABPP, FGQU, ElMiri, PP, WW, RHTM}. Surprisingly enough,
so far we were not able to find previous results dealing with Neumann
boundary conditions. This case is considered only when the absorption term
is cancelled, see \cite{GMN, Piava}. We also mention \cite{ACF, AM, BE,
CLMous, DM, GMM, MMZ} where convective problem $\left( \mathrm{P}\right) $
(without absorption) subjected to Dirichlet boundary conditions is examined.

Problem $\left( \mathrm{P}\right) $ exhibits interesting features resulting
from the interaction between absorption and reaction-convection terms. Their
involvement in $\left( \mathrm{P}\right) $ give rise to nontrivial
difficulties such as the loss of variational structure thereby making it
impossible applying variational methods. Obviously, the mere fact of their
presence impacts substantially the structure of $\left( \mathrm{P}\right) $
as well as the nature of its solutions which, in some cases, leads to
surprising situations, especially from a mathematical point of view. For
instance, in \cite{APP, PP}, it is shown that the absorption term
regularizes solutions and it is sufficient to break down any resonant effect
of the reaction term. In \cite{BG}, it is established that a problem admits
nontrivial weak solutions only under the effect of absorption. Otherwise,
zero is the only solution for the problem.

In the present paper, we provide at least two nontrivial solutions for
problem $(\mathrm{P})$ with precise sign properties: one is nodal (i.e.,
sign-changing) and the other is positive. According to our knowledge, this
topic is a novelty. The study of the existence of nodal solutions has never
been discussed for convection-absorption problems, just as the latter have
never been handled under Neumann boundary conditions.

The multiplicity result is achieved in part through a location principle of
nodal solutions which, in particular, helps distinguish between solutions of 
$(\mathrm{P})$. However, this principle constitutes in itself a crucial part
of our work since the multiplicity result depends on it. Indeed, under
assumption $(\mathrm{H.}2)$ (cf. section \ref{S2}), if $f(x,0,0)\geq 0$ a.e. in $\Omega $, it is shown that every nodal solution of problem $(\mathrm{P})$
should be bounded above by a positive solution. In particular, this provides the powerful fact that the existence of a nodal solution implies under the stated hypotheses that a positive solution must exists. In other words, nodal solutions generate positive solutions which is an unusual fact since generally, it is rather the opposite implication occurs (see \cite{CM, Mot, MotPap}). This phenomenon happens also in the opposite unilateral sens:
if $f(x,0,0)\leq 0$ a.e. in $\Omega $, every nodal solution of problem $(%
\mathrm{P})$ should be bounded below by a negative solution. However, if $%
f(x,0,0)=0$ a.e. in $\Omega $, the both cases above are satisfied
simultaneously and hence, every nodal solution to problem $(\mathrm{P})$ is
between two opposite constant-sign solutions.

The location principle of nodal solutions is stated in Theorem \ref{T2}.\
The proof is chiefly based on Theorem \ref{T3}, shown in Section \ref{S3}
via monotone operator theory together with perturbation argument and
adequate truncation. Theorem \ref{T3} is a version of sub-supersolutions
result for quasilinear convective elliptic problems involving natural
growth. It can be applied for large classes of Neumann elliptic problems
since no sign condition on the nonlinearities is required and no specific
structure is imposed. However, it is worth noting that due to the effect of
the presence of the absorption term, stretching out monotone operators
theory's scope to convection-absorption problems is not a straightforward
task. This requires, on the one hand, truncation in order to stay inside the
rectangle formed by sub-supersolution pair and, on the other hand,
perturbation (regularization), by introducing a parameter $\varepsilon >0$
in $(\mathrm{P})$, necessary to have a minimal control on the absorption
term.

Another significant feature of our result lies in obtaining nodal solutions
for problem $(\mathrm{P})$. Taking advantage of Theorem \ref{T3}, we
construct a sign-changing sub-supersolution pair $(\underline{u},\overline{u}%
) $ for problem $(\mathrm{P})$ which inevitably leads to a nodal solution $%
u_{0}$ for $(\mathrm{P})$.\ The choice of suitable functions with an
adjustment of adequate constants closely dependent on the small parameter $%
\delta >0$ is crucial. By construction, the subsolution $\underline{u}$ is
positive inside the domain $\Omega $ while the supersolution $\overline{u}$
is negative near the boundary $\partial \Omega $. Therefore, the solution $%
u_{0}$ of $(\mathrm{P}),$ being naturally imbued with these properties, is
positive inside $\Omega $ and negative once $d(x)\rightarrow 0$. We
emphasize that, without the implication of the absorption term, it would not have been possible to get nodal solutions for problem $(\mathrm{P}),$ at least with the techniques
developed in this work.

The rest of the paper is organized as follows. Section \ref{S3} contains the
existence theorem involving sub-supersolutions. Section \ref{S2} focuses on
a location principle of nodal solutions. Section \ref{S4} deals with the
multiplicity result.

\section{A sub-supersolution theorem}

\label{S3}

In the sequel, the Banach space $W^{1,p}(\Omega )$ is equipped with the
following usual norm%
\begin{equation*}
\Vert u\Vert _{1,p}:=\left( \Vert u\Vert _{p}^{p}+\Vert \nabla u\Vert
_{p}^{p}\right) ^{1/p}\text{,\quad }u\in W^{1,p}(\Omega )\text{,}
\end{equation*}%
where, as usual, 
\begin{equation*}
\Vert v\Vert _{p}:=\left\{ 
\begin{array}{l}
\left( \int_{\Omega }|v(x)|^{p}dx\right) ^{1/p}\text{ if }p<+\infty , \\ 
ess\underset{x\in \Omega }{\sup }\text{\thinspace }|v(x)|\text{ otherwise.}%
\end{array}%
\right.
\end{equation*}

The following assumptions will be posited.

\begin{itemize}
\item[$(\mathrm{H.}1)$] Let $0\leq q\leq p-1$. For every $\rho >0$ there
exists $M:=M(\rho )>0$ such that 
\begin{equation*}
|f(x,s,\xi )|\leq M(1+|\xi |^{q})\text{ \ in }\Omega \times \lbrack -\rho
,\rho ]\times \mathbb{R}^{N}.
\end{equation*}

\item[$(\mathrm{H.}2)$] There are $\underline{u},\overline{u}\in \mathcal{C}%
^{1}(\overline{\Omega })$ fulfilling%
\begin{equation}
\overline{u}+\delta \geq \underline{u}+\delta >0\text{ \ a.e. in }\Omega ,
\label{c3}
\end{equation}%
as well as 
\begin{equation}
\left\{ 
\begin{array}{l}
\int_{\Omega }|\nabla \underline{u}|^{p-2}\nabla \underline{u}\,\nabla
\varphi \text{\thinspace }dx+\int_{\Omega }\frac{|\nabla \underline{u}|^{p}}{%
\underline{u}+\delta }\varphi \,dx-\int_{\Omega }f(x,\underline{u},\nabla 
\underline{u})\varphi \,dx\leq 0, \\ 
\int_{\Omega }|\nabla \overline{u}|^{p-2}\nabla \overline{u}\,\nabla \varphi
\,dx+\int_{\Omega }\frac{|\nabla \overline{u}|^{p}}{\overline{u}+\delta }%
\varphi \,dx-\int_{\Omega }f(x,\overline{u},\nabla \overline{u})\varphi
\,dx\geq 0,%
\end{array}%
\right.  \label{c2}
\end{equation}%
for all $\varphi \in W^{1,p}(\Omega )\cap L^{\infty }(\Omega )$ with $%
\varphi \geq 0$ in $\Omega $.
\end{itemize}

The functions $\underline{u}$ and $\overline{u}$ in $(\mathrm{H.}2)$ are
called subsolution and supersolution of problem $(\mathrm{P})$, respectively.

\subsection{An auxiliary problem}

Let $\underline{u},\overline{u}\in \mathcal{C}^{1}(\overline{\Omega })$ be a
sub-supersolutions of problem $(\mathrm{P})$ as required in condition $(%
\mathrm{H.}2)$. We consider the truncation operators $\mathcal{T}%
:W^{1,p}(\Omega )\rightarrow W^{1,p}(\Omega )$ defined by%
\begin{equation}
\mathcal{T}(u):=\left\{ 
\begin{array}{ll}
\underline{u} & \text{when }u\leq \underline{u}, \\ 
u & \text{if }\underline{u}\leq u\leq \overline{u}, \\ 
\overline{u} & \text{otherwise.}%
\end{array}%
\right.  \label{33}
\end{equation}%
Lemma 2.89 in \cite{CLM} ensures that $\mathcal{T}$ is continuous and
bounded. We introduce the cut-off function $b:\Omega \times \mathbb{R}%
\longrightarrow \mathbb{R}$ defined by%
\begin{equation*}
b(x,s):=-(\underline{u}(x)-s)_{+}^{p-1}+(s-\overline{u}(x))_{+}^{p-1},\quad
(x,s)\in \Omega \times \mathbb{R},
\end{equation*}%
The function $b$ is a Carath\'{e}odory function satisfying the growth
condition 
\begin{equation}
|b(x,s)|\leq k(x)+c|s|^{p-1},\ \text{for a.a. }x\in \Omega \text{, for all }%
s\in 
\mathbb{R}
,  \label{growth}
\end{equation}%
where $c\geq 0$ is a positive constant and $k$ $\in L^{\infty }(\Omega ).$
Moreover, it holds%
\begin{equation}
\int_{\Omega }b(\cdot ,u)u\,\mathrm{d}x\geq C_{1}\Vert u\Vert _{p}^{p}-C_{2},%
\text{ for all }u\in W^{1,p}(\Omega ),  \label{35}
\end{equation}%
with appropriate constants $C_{1},C_{2}>0$; see, e.g., \cite[pp. 95--96]{CLM}%
.

For $\varepsilon \in (0,1)$ and for $\mu >0$ that will be selected later on,
we state the auxiliary problem%
\begin{equation*}
(\mathrm{P}_{\varepsilon ,\mu })\qquad \left\{ 
\begin{array}{ll}
-\Delta _{p}{u+}\frac{|\nabla (\mathcal{T}u)|^{p}}{\mathcal{T}u+\delta
+\varepsilon }=f(x,\mathcal{T}u,\nabla (\mathcal{T}u))-\mu b(x,u) & \text{in}%
\;\Omega , \\ 
|\nabla u|^{p-2}\frac{\partial u}{\partial \eta }=0 & \text{on}\;\partial
\Omega .%
\end{array}%
\right.
\end{equation*}%
We provide the existence of solutions $u\in W^{1,p}(\Omega )$ for problem $(%
\mathrm{P}_{\varepsilon ,\mu })$. The proof is chiefly based on
pseudomonotone operators theorem stated in \cite[Theorem 2.99]{CLM}.

Next lemmas furnish useful estimates related to nonlinear terms involved in $%
(\mathrm{P}_{\mu })$. The first estimate deals with the nonlinearity $f$.

\begin{lemma}
\label{L3}Under assumption $(\mathrm{H.}1)$ there exists a constant $C_{0}>0$
such that, for all $u\in W^{1,p}(\Omega ),$ we have%
\begin{equation*}
\int_{\Omega }|f(\cdot ,\mathcal{T}u,\nabla (\mathcal{T}u))||u|\mathrm{d}%
x\leq \frac{1}{2}\Vert \nabla u\Vert _{p}^{p}+C_{0}(1+\left\Vert
u\right\Vert _{p}+\left\Vert u\right\Vert _{p}^{p}).
\end{equation*}
\end{lemma}

\begin{proof}
For any fixed $\sigma \in ]0,\frac{1}{2M}[$, Young's inequality implies 
\begin{equation}
\begin{array}{ll}
|\nabla (\mathcal{T}u)|^{q}|u| & \leq \sigma |\nabla (\mathcal{T}u)|^{\frac{%
qp}{p-1}}+c_{\sigma }|u|^{p} \\ 
& \leq \sigma (1+|\nabla (\mathcal{T}u)|^{p})+c_{\sigma }|u|^{p},%
\end{array}
\label{28*}
\end{equation}%
for every $u\in W^{1,p}(\Omega )$. Moreover, by (\ref{33}) we have%
\begin{equation*}
\begin{array}{c}
\int_{\Omega }|\nabla (\mathcal{T}u)|^{p}\text{ }\mathrm{d}x=\int_{\{u\leq 
\underline{u}\}}|\nabla \underline{u}|^{p}\text{ }\mathrm{d}x+\int_{\{%
\underline{u}\leq u\leq \overline{u}\}}|\nabla u|^{p}\text{ }\mathrm{d}%
x+\int_{\{u\geq \overline{u}\}}|\nabla \overline{u}|^{p}\text{ }\mathrm{d}x
\\ 
\leq \int_{\Omega }|\nabla \underline{u}|^{p}\text{ }\mathrm{d}%
x+\int_{\Omega }|\nabla u|^{p}\text{ }\mathrm{d}x+\int_{\Omega }|\nabla 
\overline{u}|^{p}\text{ }\mathrm{d}x \\ 
\leq |\Omega |(\left\Vert \nabla \underline{u}\right\Vert _{\infty
}^{p}+\left\Vert \nabla \overline{u}\right\Vert _{\infty }^{p})+\Vert \nabla
u\Vert _{p}^{p}.%
\end{array}%
\end{equation*}%
Then, using $(\mathrm{H.}1),$ (\ref{28*}) and the fact that $\sigma <\frac{1%
}{2M}$, thanks to H\"{o}lder's inequality, we get%
\begin{equation}
\begin{array}{l}
\int_{\Omega }|f(\cdot ,\mathcal{T}u,\nabla (\mathcal{T}u))||u|\text{ }%
\mathrm{d}x\leq M\int_{\Omega }(1+|\nabla (\mathcal{T}u)|^{q})|u|\text{ }%
\mathrm{d}x \\ 
\leq M\int_{\Omega }(|u|+\sigma (1+|\nabla (\mathcal{T}u)|^{p})+c_{\sigma
}|u|^{p})\text{ }\mathrm{d}x \\ 
\leq M(|\Omega |^{\frac{p-1}{p}}\left\Vert u\right\Vert _{p}+\sigma |\Omega
|(1+\left\Vert \nabla \underline{u}\right\Vert _{\infty }^{p}+\left\Vert
\nabla \overline{u}\right\Vert _{\infty }^{p})+\sigma \Vert \nabla u\Vert
_{p}^{p}+c_{\sigma }\left\Vert u\right\Vert _{p}^{p}) \\ 
\leq \frac{1}{2}\Vert \nabla u\Vert _{p}^{p}+C_{0}(1+\left\Vert u\right\Vert
_{p}+\left\Vert u\right\Vert _{p}^{p}),%
\end{array}
\label{37}
\end{equation}%
which completes the proof.
\end{proof}

We turn to estimating the natural growth gradient term in $(\mathrm{P}%
_{\varepsilon ,\mu })$.

\begin{lemma}
\label{L4}Assume that $(\mathrm{H.}1)$ and $(\mathrm{H.}2)$ hold. Then, for
all $u\in W^{1,p}(\Omega )$, there exists a constant $\hat{C}_{\varepsilon
}>0$, independent of $u$, such that%
\begin{equation}
\begin{array}{l}
\int_{\Omega }\frac{|\nabla (\mathcal{T}u)|^{p}}{\mathcal{T}u+\delta
+\varepsilon }|u|\text{ }\mathrm{d}x\leq \hat{C}_{\varepsilon }(1+\left\Vert
u\right\Vert _{p})\text{, for all }\varepsilon \in (0,1).%
\end{array}
\label{17*}
\end{equation}
\end{lemma}

\begin{proof}
By (\ref{33}) note that%
\begin{equation*}
\mathcal{T}u=\underline{u}\mathbbm{1}_{\{u\leq \underline{u}\}}+u\mathbbm{1}%
_{\{\underline{u}<u<\overline{u}\}}+\overline{u}\mathbbm{1}_{\{u\geq 
\overline{u}\}},\text{ for }u\in W^{1,p}(\Omega ).
\end{equation*}%
Then%
\begin{eqnarray*}
&&\int_{\Omega }\frac{|\nabla (\mathcal{T}u)|^{p}}{\mathcal{T}u+\delta
+\varepsilon }|u|\text{ }\mathrm{d}x \\
&=&\int_{\{u\leq \underline{u}\}}\frac{|\nabla \underline{u}|^{p}}{%
\underline{u}+\delta +\varepsilon }|u|\text{ }\mathrm{d}x+\int_{\Omega }%
\frac{|\nabla (u\mathbbm{1}_{\{\underline{u}<u<\overline{u}\}})|^{p}}{u%
\mathbbm{1}_{\{\underline{u}<u<\overline{u}\}}+\delta +\varepsilon }|u|\text{
}\mathrm{d}x+\int_{\{u\geq \overline{u}\}}\frac{|\nabla \overline{u}|^{p}}{%
\overline{u}+\delta +\varepsilon }|u|\text{ }\mathrm{d}x.
\end{eqnarray*}%
In view of (\ref{c3}) we have%
\begin{equation*}
\overline{u}+\delta +\varepsilon \geq \underline{u}+\delta +\varepsilon
>\varepsilon \text{ \ a.e. in }\Omega .
\end{equation*}%
Hence%
\begin{equation*}
\begin{array}{l}
\int_{\{u\leq \underline{u}\}}\frac{|\nabla \underline{u}|^{p}}{\underline{u}%
+\delta +\varepsilon }|u|\text{ }\mathrm{d}x\leq \frac{\left\Vert \nabla 
\underline{u}\right\Vert _{\infty }^{p}}{\varepsilon }\int_{\{u\leq 
\underline{u}\}}|u|\text{ }\mathrm{d}x\leq \frac{\left\Vert \nabla 
\underline{u}\right\Vert _{\infty }^{p}}{\varepsilon }\int_{\Omega }|u|\text{
}\mathrm{d}x\leq \frac{\left\Vert \nabla \underline{u}\right\Vert _{\infty
}^{p}}{\varepsilon }|\Omega |^{\frac{p-1}{p}}\left\Vert u\right\Vert _{p},%
\end{array}%
\end{equation*}%
\begin{equation*}
\begin{array}{l}
\int_{\{u\geq \overline{u}\}}\frac{|\nabla \overline{u}|^{p}}{\overline{u}%
+\delta +\varepsilon }|u|\text{ }\mathrm{d}x\leq \frac{\left\Vert \nabla 
\overline{u}\right\Vert _{\infty }^{p}}{\varepsilon }\int_{\{u\geq \overline{%
u}\}}|u|\text{ }\mathrm{d}x\leq \frac{\left\Vert \nabla \overline{u}%
\right\Vert _{\infty }^{p}}{\varepsilon }\int_{\Omega }|u|\text{ }\mathrm{d}%
x\leq \frac{\left\Vert \nabla \overline{u}\right\Vert _{\infty }^{p}}{%
\varepsilon }|\Omega |^{\frac{p-1}{p}}\left\Vert u\right\Vert _{p}%
\end{array}%
\end{equation*}%
and%
\begin{eqnarray*}
\int_{\Omega }\frac{|\nabla (u\mathbbm{1}_{\{\underline{u}<u<\overline{u}%
\}})|^{p}}{u\mathbbm{1}_{\{\underline{u}<u<\overline{u}\}}+\delta
+\varepsilon }|u|\text{ }\mathrm{d}x &\leq &\frac{\max \{|\underline{u}|,|%
\overline{u}|\}}{\varepsilon }\int_{\Omega }|\nabla (u\mathbbm{1}_{\{%
\underline{u}<u<\overline{u}\}})|^{p}\text{ }\mathrm{d}x \\
&\leq &\frac{\max \{\left\Vert \underline{u}\right\Vert _{\infty
},\left\Vert \overline{u}\right\Vert _{\infty }\}}{\varepsilon }\int_{\Omega
}|\nabla (u\mathbbm{1}_{\{\underline{u}<u<\overline{u}\}})|^{p}\text{ }%
\mathrm{d}x \\
&\leq &\frac{\max \{\left\Vert \underline{u}\right\Vert _{\infty
},\left\Vert \overline{u}\right\Vert _{\infty }\}}{\varepsilon }\Vert u%
\mathbbm{1}_{\{\underline{u}<u<\overline{u}\}}\Vert _{1,p}^{p}.
\end{eqnarray*}%
Gathering the above inequalities we obtain%
\begin{equation}
\begin{array}{l}
\int_{\Omega }\frac{|\nabla (\mathcal{T}u)|^{p}}{\mathcal{T}u+\delta
+\varepsilon }|u|\text{ }\mathrm{d}x\leq \frac{\left\Vert \nabla \underline{u%
}\right\Vert _{\infty }^{p}+\left\Vert \nabla \overline{u}\right\Vert
_{\infty }^{p}}{\varepsilon }|\Omega |^{\frac{p-1}{p}}\left\Vert
u\right\Vert _{p}+\frac{\max \{\left\Vert \underline{u}\right\Vert _{\infty
},\left\Vert \overline{u}\right\Vert _{\infty }\}}{\varepsilon }\Vert u%
\mathbbm{1}_{\{\underline{u}<u<\overline{u}\}}\Vert _{1,p}^{p}.%
\end{array}
\label{17}
\end{equation}%
We claim that $\Vert u\mathbbm{1}_{\{\underline{u}<u<\overline{u}\}}\Vert
_{p}$ is uniformly bounded. Indeed, test in $(\mathrm{P}_{\varepsilon ,\mu
}) $ with $(u+\delta )\mathbbm{1}_{\{\underline{u}<u<\overline{u}\}}\in
W^{1,p}(\Omega )\cap L^{\infty }(\Omega )$ which is possible in view of \cite%
[Proposition 1.61]{MMPA}. Here, on the basis of (\ref{growth}) and $(\mathrm{%
H.}1),$ with $-\rho \leq \underline{u}\leq \overline{u}\leq \rho $, for $%
u\in \lbrack \underline{u},\overline{u}],$ $\underline{u},\overline{u}\in
L^{\infty }(\Omega )$ (see $(\mathrm{H.}2)$), one has%
\begin{equation*}
\begin{array}{l}
\left\vert f(x,\mathcal{T}u,\nabla (\mathcal{T}u))-\mu b(x,u)-\frac{|\nabla (%
\mathcal{T}u)|^{p}}{\mathcal{T}u+\delta +\varepsilon }\right\vert \\ 
\leq |f(x,\mathcal{T}u,\nabla (\mathcal{T}u))|+\mu |b(x,u)|+\frac{|\nabla (%
\mathcal{T}u)|^{p}}{\mathcal{T}u+\delta +\varepsilon } \\ 
\leq C_{\varepsilon }(1+|\nabla (\mathcal{T}u)|^{q}+|\nabla (\mathcal{T}%
u)|^{p}),%
\end{array}%
\end{equation*}%
for a certain constant $C_{\varepsilon }>0$ independent of $u$. Then, the
regularity up to the boundary result in \cite{L} ensures that $u\in \mathcal{%
C}^{1,\tau }(\overline{\Omega })$ for certain $\tau \in (0,1).$ Therefore, 
\cite[Proposition 1.61]{MMPA} applies.

Then, noting that 
\begin{equation}
b(x,u)=0\text{ \ a.e. for }u\in \lbrack \underline{u},\overline{u}],
\label{20}
\end{equation}%
it follows that%
\begin{equation}
\begin{array}{l}
\int_{\Omega }|\nabla (u\mathbbm{1}_{\{\underline{u}<u<\overline{u}\}})|^{p}%
\text{ }\mathrm{d}x+\int_{\Omega }\frac{|\nabla (\mathcal{T}u)|^{p}}{%
\mathcal{T}u+\delta +\varepsilon }(u+\delta )\mathbbm{1}_{\{\underline{u}<u<%
\overline{u}\}}\text{ }\mathrm{d}x \\ 
=\int_{\Omega }f(x,\mathcal{T}u,\nabla (\mathcal{T}u))(u+\delta )\mathbbm{1}%
_{\{\underline{u}<u<\overline{u}\}}\text{ }\mathrm{d}x.%
\end{array}
\label{18}
\end{equation}

Due to (\ref{c3}) and (\ref{33}) one has%
\begin{equation*}
\int_{\Omega }\frac{|\nabla (\mathcal{T}u)|^{p}}{\mathcal{T}u+\delta
+\varepsilon }(u+\delta )\mathbbm{1}_{\{\underline{u}<u<\overline{u}\}}\text{
}\mathrm{d}x\geq 0.
\end{equation*}%
Therefore, from (\ref{18}), we deduce that 
\begin{equation}
\begin{array}{c}
\int_{\Omega }|\nabla (u\mathbbm{1}_{\{\underline{u}<u<\overline{u}\}})|^{p}%
\text{ }\mathrm{d}x\leq \int_{\Omega }f(x,\mathcal{T}u,\nabla (\mathcal{T}%
u))(u+\delta )\mathbbm{1}_{\{\underline{u}<u<\overline{u}\}}\text{ }\mathrm{d%
}x \\ 
\leq \int_{\Omega }|f(x,\mathcal{T}u,\nabla (\mathcal{T}u))|(|u|+\delta )%
\mathbbm{1}_{\{\underline{u}<u<\overline{u}\}}\text{ }\mathrm{d}x.%
\end{array}
\label{29*}
\end{equation}%
Exploiting (\ref{37}) and $(\mathrm{H.}1)$ we get%
\begin{equation*}
\begin{array}{l}
\int_{\Omega }|f(x,\mathcal{T}u,\nabla (\mathcal{T}u))|(|u|+\delta )%
\mathbbm{1}_{\{\underline{u}<u<\overline{u}\}}\text{ }\mathrm{d}x \\ 
=\int_{\Omega }|f(x,\mathcal{T}u,\nabla (\mathcal{T}u))||u|\mathbbm{1}_{\{%
\underline{u}<u<\overline{u}\}}\text{ }\mathrm{d}x+\delta \int_{\Omega }|f(x,%
\mathcal{T}u,\nabla (\mathcal{T}u))|\mathbbm{1}_{\{\underline{u}<u<\overline{%
u}\}}\text{ }\mathrm{d}x \\ 
\leq \frac{1}{2}\Vert \nabla (u\mathbbm{1}_{\{\underline{u}<u<\overline{u}%
\}})\Vert _{p}^{p}+C(1+\left\Vert u\mathbbm{1}_{\{\underline{u}<u<\overline{u%
}\}}\right\Vert _{p}+\left\Vert u\mathbbm{1}_{\{\underline{u}<u<\overline{u}%
\}}\right\Vert _{p}^{p}) \\ 
+\delta M(1+\Vert \nabla (u\mathbbm{1}_{\{\underline{u}<u<\overline{u}%
\}})\Vert _{p}^{q}) \\ 
\leq \frac{1}{2}\Vert \nabla (u\mathbbm{1}_{\{\underline{u}<u<\overline{u}%
\}})\Vert _{p}^{p}+\delta M\Vert \nabla (u\mathbbm{1}_{\{\underline{u}<u<%
\overline{u}\}})\Vert _{p}^{q}+\tilde{C}_{0},%
\end{array}%
\end{equation*}%
where $\tilde{C}_{0}:=C(1+\rho |\Omega |^{\frac{1}{p}}+\rho ^{p}|\Omega
|)+\delta M$. Combining with (\ref{29*}) and since $q<p$, we conclude that
there is a constant $\tilde{C}>0,$ independent of $u,$ such that 
\begin{equation}
\Vert \nabla (u\mathbbm{1}_{\{\underline{u}<u<\overline{u}\}})\Vert _{p}\leq 
\tilde{C}.  \label{30}
\end{equation}%
This proves the claim.

Consequently, in view of (\ref{17}) and (\ref{30}), we infer that there
exists a constant $\hat{C}_{\varepsilon }>0$, independent of $u$, such that (%
\ref{17*}) holds true. This ends the proof.
\end{proof}

The existence result for problem $(\mathrm{P}_{\varepsilon ,\mu })$ is
formulated as follows.

\begin{theorem}
\label{T5} Suppose $(\mathrm{H.}1)$--$(\mathrm{H.}2)$ hold true. Then,
problem $(\mathrm{P}_{\mu ,\varepsilon })$ possesses a weak solution $u\in
W^{1,p}(\Omega )$ for all $\varepsilon \in (0,1)$.
\end{theorem}

\begin{proof}
By (\ref{growth}), the Nemytskii operator $\mathcal{B}$ given by $\mathcal{B}%
u(x)=b(\cdot ,u)$ is well defined and $\mathcal{B}:W^{1,p}(\Omega
)\longrightarrow W^{-1,p^{\prime }}(\Omega )$ is continuous and bounded. By
the compact embedding $W^{1,p}(\Omega )\hookrightarrow L^{p}(\Omega )$, $%
\mathcal{B}$ is completely continuous.

Considering (\ref{c3}), define the function $\pi _{\delta ,\varepsilon
}:(-\delta ,+\infty )\times 
\mathbb{R}
^{N}\longrightarrow 
\mathbb{R}
$ by%
\begin{equation*}
\pi _{\delta }(s,\xi )=\frac{|\xi |^{p}}{s+\delta +\varepsilon }
\end{equation*}%
which satisfies the estimate%
\begin{equation*}
|\pi _{\delta }(s,\xi )|\leq \frac{1}{\varepsilon }|\xi |^{p},\text{ for all 
}s>-\delta \text{, }\xi \in 
\mathbb{R}
^{N}\text{ and all }\varepsilon \in (0,1).
\end{equation*}%
Let $\Pi _{\delta ,\varepsilon }:[\underline{u},\overline{u}]\subset
W^{1,p}(\Omega )\longrightarrow L^{1}(\Omega )\subset W^{-1,p^{\prime
}}(\Omega )$ denotes the corresponding Nemytskii operator, that is $\Pi
_{\delta ,\varepsilon }u(x)=\pi _{\delta ,\varepsilon }(u(x),\nabla u(x)),$
which is bounded and continuous (see \cite[Theorem 2.76]{MMPA} and \cite[%
Theorem 3.4.4]{GP}). Moreover, $\Pi _{\delta ,\varepsilon }$ is completely
continuous due to the compact embedding of $W^{1,p}(\Omega )$ into $%
L^{p}(\Omega )$.

In view of $(\mathrm{H.}1),$ if $\rho >0$ satisfies 
\begin{equation}
-\rho \leq \underline{u}\leq \overline{u}\leq \rho ,  \label{14}
\end{equation}%
the Nemitskii operator $\mathcal{N}_{f}:[\underline{u},\overline{u}]\subset
W^{1,p}(\Omega )\rightarrow W^{-1,p^{\prime }}(\Omega )$ generated by the
Carath\'{e}odory function $f$ is bounded and completely continuous thanks to
Rellich-Kondrachov compactness embedding theorem.

At this point, problem $(\mathrm{P}_{\mu ,\varepsilon })$ can be
equivalently expressed as%
\begin{equation}
\mathcal{A}_{\mu ,\varepsilon }(u):=-\Delta _{p}u+\mu \mathcal{B}u+\Pi
_{\delta ,\varepsilon }\circ \mathcal{T}(u)-\mathcal{N}_{f}\circ \mathcal{T}%
(u)=0\text{ \ in }W^{-1,p^{\prime }}(\Omega ).  \label{operator1}
\end{equation}%
By $(\mathrm{H.}1),$ it is readily seen that the operator $\mathcal{A}_{\mu
,\varepsilon }:W^{1,p}(\Omega )\rightarrow W^{-1,p^{\prime }}(\Omega )$ is
well defined, bounded and continuous.

Let us show that $\mathcal{A}_{\mu ,\varepsilon }$ is coercive. From (\ref%
{operator1}) we have%
\begin{equation}
\begin{array}{l}
\left\langle \mathcal{A}_{\mu ,\varepsilon }(u),u\right\rangle =\int_{\Omega
}|\nabla u|^{p}\text{ }\mathrm{d}x+\mu \int_{\Omega }b(x,u)u\text{ }\mathrm{d%
}x+\int_{\Omega }\frac{|\nabla (\mathcal{T}u)|^{p}}{\mathcal{T}u+\delta
+\varepsilon }u\text{ }\mathrm{d}x-\int_{\Omega }f(\cdot ,\mathcal{T}%
u,\nabla (\mathcal{T}u))u\,\mathrm{d}x \\ 
\geq \int_{\Omega }|\nabla u|^{p}\text{ }\mathrm{d}x+\mu \int_{\Omega
}b(x,u)u\text{ }\mathrm{d}x-\int_{\Omega }\frac{|\nabla (\mathcal{T}u)|^{p}}{%
\mathcal{T}u+\delta +\varepsilon }|u|\text{ }\mathrm{d}x-\int_{\Omega
}f(\cdot ,\mathcal{T}u,\nabla (\mathcal{T}u))u\,\mathrm{d}x.%
\end{array}
\label{16}
\end{equation}%
Bearing in mind (\ref{35}) as well as the estimates in Lemmas \ref{L3} and %
\ref{L4}, we thus arrive at%
\begin{equation*}
\begin{array}{l}
\left\langle \mathcal{A}_{\mu ,\varepsilon }(u),u\right\rangle \geq \Vert
\nabla u\Vert _{p}^{p}+\mu \left( C_{1}\Vert u\Vert _{p}^{p}-C_{2}\right) \\ 
-\hat{C}_{\varepsilon }(1+\left\Vert u\right\Vert _{p})-\frac{1}{2}\Vert
\nabla u\Vert _{p}^{p}-C_{0}(1+\left\Vert u\right\Vert _{p}+\left\Vert
u\right\Vert _{p}^{p}).%
\end{array}%
\end{equation*}%
In view of (\ref{30}) and for $\mu >0$ large so that $\mu C_{1}-C_{0}>0,$
for every sequence $(u_{n})_{n}$ in $W^{1,p}(\Omega )$, the last inequality
forces 
\begin{equation*}
\lim_{n\rightarrow +\infty }\frac{\langle \mathcal{A}_{\mu ,\varepsilon
}(u_{n}),u_{n}\rangle }{\Vert u_{n}\Vert _{1,p}}=+\infty ,
\end{equation*}%
as desired.

The next step is to show that the operator $\mathcal{A}_{\mu }$ is
pseudomonotone. Toward this, suppose $u_{n}\rightharpoonup u$ in $%
W^{1,p}(\Omega )$ and 
\begin{equation*}
\limsup_{n\rightarrow +\infty }\langle \mathcal{A}_{\mu ,\varepsilon
}(u_{n}),u_{n}-u\rangle \leq 0.
\end{equation*}%
In view of the complete continuity of the operators $\mathcal{B}$, $\Pi
_{\delta ,\varepsilon }$ and $\mathcal{N}_{f}$, we get 
\begin{eqnarray*}
\underset{n\rightarrow \infty }{\lim }\langle \mathcal{B}(u_{n}),u_{n}-u%
\rangle &=&0, \\
\underset{n\rightarrow \infty }{\lim }\langle \Pi _{\delta ,\varepsilon
}(u_{n}),u_{n}-u\rangle &=&0, \\
\underset{n\rightarrow \infty }{\lim }\langle \mathcal{N}_{f}(\mathcal{T}%
u_{n}),u_{n}-u\rangle &=&0.
\end{eqnarray*}%
Then, using the $(\mathrm{S})_{+}$-property of $-\Delta _{p}$, we deduce
that $u_{n}\rightarrow u$ in $W^{1,p}(\Omega ).$ Therefore,%
\begin{equation*}
\underset{n\rightarrow \infty }{\lim }\langle \mathcal{A}_{\mu ,\varepsilon
}(u_{n}),u_{n}-v\rangle =\langle \mathcal{A}_{\mu ,\varepsilon
}(u),u-v\rangle ,
\end{equation*}%
for all $v\in W^{1,p}(\Omega ),$ because $\mathcal{A}_{\mu ,\varepsilon }$
is continuous. This proves that the operator $\mathcal{A}_{\mu ,\varepsilon
} $ is pseudomonotone.

According to the properties above, we are in a position to apply the main
theorem for pseudomonotone operators \cite[Theorem 2.99]{CLM} to the
operator $\mathcal{A}_{\mu ,\varepsilon }$. It entails the existence of $%
u\in W^{1,p}(\Omega )$ fulfilling%
\begin{equation*}
\left\langle \mathcal{A}_{\mu ,\varepsilon }(u),\varphi \right\rangle =0,\ \
\varphi \in W^{1,p}(\Omega ).
\end{equation*}%
Owing to \cite[Theorem 3]{CF}, one has 
\begin{equation*}
|\nabla u|^{p-2}\frac{\partial u}{\partial \eta }=0\ \text{on }\partial
\Omega \text{.}
\end{equation*}%
Thus, $u\in W^{1,p}(\Omega )$ is a weak solution of $(\mathrm{P}_{\mu
,\varepsilon })$. This ends the proof.
\end{proof}

\subsection{A sub-supersolution Theorem}

\begin{theorem}
\label{T3} Suppose $(\mathrm{H.}1)$--$(\mathrm{H.}2)$ hold true. Then,
problem $(\mathrm{P})$ possesses a solution $u\in \mathcal{C}^{1}(\overline{%
\Omega })$ such that 
\begin{equation}
\underline{u}\leq u\leq \overline{u}.  \label{19}
\end{equation}
\end{theorem}

\begin{proof}
According to Theorem \ref{T5}, problem $(\mathrm{P}_{\mu ,\varepsilon })$
admits a weak solution $u$ in $W^{1,p}(\Omega ).$ Let us next verify$\ $that 
$u$ satisfy the inequalities (\ref{19}). We provide the argument only for $%
u\leq \overline{u}$ because $\underline{u}\leq u$ can be similarly
established. First, note from Lemma \ref{L4} that $\frac{|\nabla (\mathcal{T}%
u)|^{p}}{\mathcal{T}u+\delta +\varepsilon }(u-\overline{u})_{+}\in
L^{1}(\Omega )$. Thus, test $(\mathrm{P}_{\mu ,\varepsilon })$ with $(u-%
\overline{u})_{+}\in W^{1,p}(\Omega )$ and taking $(\mathrm{H.}2)$ into
account, we achieve 
\begin{equation*}
\begin{array}{l}
\int_{\Omega }|\nabla u|^{p-2}\nabla u\text{\thinspace }\nabla (u-\overline{u%
})_{+}\text{ }\mathrm{d}x+\int_{\Omega }\frac{|\nabla (\mathcal{T}u)|^{p}}{%
\mathcal{T}u+\delta +\varepsilon }(u-\overline{u})_{+}\text{ }\mathrm{d}x \\ 
=\int_{\Omega }f(\cdot ,\mathcal{T}u,\nabla (\mathcal{T}u))(u-\overline{u}%
)_{+}\text{ }\mathrm{d}x-\mu \int_{\Omega }b(\cdot ,u)(u-\overline{u})_{+}%
\text{ }\mathrm{d}x \\ 
=\int_{\Omega }f(\cdot ,\overline{u},\nabla \overline{u})(u-\overline{u})_{+}%
\text{ }\mathrm{d}x-\mu \int_{\Omega }(u-\overline{u})_{+}^{p}\text{ }%
\mathrm{d}x \\ 
\leq \int_{\Omega }|\nabla \overline{u}|^{p-2}\nabla \overline{u}\,\nabla (u-%
\overline{u})_{+}\,\mathrm{d}x+\int_{\Omega }\frac{|\nabla \overline{u}|^{p}%
}{\overline{u}+\delta +\varepsilon }(u-\overline{u})_{+}\,\mathrm{d}x-\mu
\int_{\Omega }(u-\overline{u})_{+}^{p}\text{ }\mathrm{d}x\text{.}%
\end{array}%
\end{equation*}%
By (\ref{33}) note that 
\begin{equation*}
\int_{\Omega }\frac{|\nabla (\mathcal{T}u)|^{p}}{\mathcal{T}u+\delta
+\varepsilon }(u-\overline{u})_{+}\text{ }\mathrm{d}x=\int_{\Omega }\frac{%
|\nabla \overline{u}|^{p}}{\overline{u}+\delta +\varepsilon }(u-\overline{u}%
)_{+}\,\mathrm{d}x.
\end{equation*}%
Then, it turns out that 
\begin{equation}
\int_{\Omega }\left( |\nabla u|^{p-2}\nabla u-|\nabla \overline{u}%
|^{p-2}\nabla \overline{u}\right) \nabla (u-\overline{u})_{+}\,\mathrm{d}%
x\leq -\mu \int_{\Omega }(u-\overline{u})_{+}^{p}\text{ }\mathrm{d}x\leq 0.
\end{equation}%
The monotonicity of $\Delta _{p}$ directly leads to $u\leq \overline{u}$.
Test $(\mathrm{P}_{\mu ,\varepsilon })$ with $(\underline{u}-u)_{+}\in
W^{1,p}(\Omega )$, a quite similar reasoning furnishes $\underline{u}\leq u$%
. Moreover, by \cite[Remark 8]{MMT}, one has $u\in \mathcal{C}^{1,\tau }(%
\overline{\Omega })$ for some $\tau \in ]0,1[$ as well as $\frac{\partial u}{%
\partial \eta }=0$ on $\partial \Omega $. Consequently, $u$ is a solution of 
$(\mathrm{P}_{\varepsilon ,\mu })$ within $[\underline{u},\overline{u}]$
which, due to (\ref{20}), reads as 
\begin{equation*}
(\mathrm{P}_{\varepsilon })\qquad \left\{ 
\begin{array}{ll}
-\Delta _{p}{u+}\frac{|\nabla (\mathcal{T}u)|^{p}}{\mathcal{T}u+\delta
+\varepsilon }=f(x,\mathcal{T}u,\nabla (\mathcal{T}u)) & \text{in}\;\Omega ,
\\ 
|\nabla u|^{p-2}\frac{\partial u}{\partial \eta }=0 & \text{on}\;\partial
\Omega .%
\end{array}%
\right.
\end{equation*}%
The task is now to find solutions of $(\mathrm{P})$ by passing to the limit
in $(\mathrm{P}_{\varepsilon })$ as $\varepsilon \rightarrow 0$. To this
end, set $\varepsilon =\frac{1}{n}$ with any positive integer $n\geq 1$,
there exists $u_{n}:=u_{\frac{1}{n}}\in \mathcal{C}^{1,\tau }(\overline{%
\Omega })$ solution of $(\mathrm{P}_{n})$ ($(\mathrm{P}_{\varepsilon })$
with $\varepsilon =\frac{1}{n}$), that is, 
\begin{equation}
u_{n}\in \lbrack \underline{u},\overline{u}]  \label{21}
\end{equation}%
and 
\begin{equation}
\begin{array}{l}
\int_{\Omega }|\nabla u_{n}|^{p-2}\nabla u_{n}\text{\thinspace }\nabla
\varphi \text{ }\mathrm{d}x+\int_{\Omega }\frac{|\nabla u_{n}|^{p}}{%
u_{n}+\delta +\frac{1}{n}}\varphi \text{ }\mathrm{d}x=\int_{\Omega
}f(x,u_{n},\nabla u_{n})\varphi \text{ }\mathrm{d}x,%
\end{array}
\label{122}
\end{equation}%
for all $\varphi \in W^{1,p}(\Omega )\cap L^{\infty }(\Omega )$. Since the
embedding $\mathcal{C}^{1,\tau }(\overline{\Omega })\subset \mathcal{C}^{1}(%
\overline{\Omega })$ is compact, we can extract subsequences (still denoted
by $\{u_{n}\}$) such that 
\begin{equation}
u_{n}\rightarrow u\text{ in }\mathcal{C}^{1}(\overline{\Omega })\text{ with }%
u\in \lbrack \underline{u},\overline{u}].  \label{22}
\end{equation}%
Therefore%
\begin{equation*}
|\nabla u_{n}|^{p-2}\nabla u_{n}\rightarrow |\nabla u|^{p-2}\nabla u\text{ \
in }\mathcal{C}(\overline{\Omega }),
\end{equation*}%
\begin{equation*}
|\nabla u_{n}|^{p}\rightarrow |\nabla u|^{p}\text{ \ in }\mathcal{C}(%
\overline{\Omega })
\end{equation*}%
and 
\begin{equation*}
f(x,u_{n},\nabla u_{n})\rightarrow f(x,u,\nabla u)\text{ \ in }\mathcal{C}(%
\overline{\Omega }),
\end{equation*}%
because $f$ is a Carath\'{e}odory function. Then, on the basis of (\ref{21})
and (\ref{c3}), owing to Lebesgue dominated convergence theorem, we may pass
to the limit as $n\rightarrow \infty $ in (\ref{122}) to get 
\begin{equation*}
\begin{array}{l}
\int_{\Omega }|\nabla u|^{p-2}\nabla u\text{\thinspace }\nabla \varphi \text{
}\mathrm{d}x+\int_{\Omega }\frac{|\nabla u|^{p}}{u+\delta }\varphi \text{ }%
\mathrm{d}x=\int_{\Omega }f(x,u,\nabla u)\varphi \text{ }\mathrm{d}x,%
\end{array}%
\end{equation*}%
for all $\varphi \in W^{1,p}(\Omega )\cap L^{\infty }(\Omega )$. Moreover,
according to (\ref{22}), we have $u+\delta >0$ a.e. in $\Omega $. This
proves that $u\in \mathcal{C}^{1}(\overline{\Omega })$ is a solution of
problem $\left( \mathrm{P}\right) $ within $[\underline{u},\overline{u}]$.
The proof is now completed.
\end{proof}

\section{Location principle for nodal solutions}

\label{S2}

It this section we focus on the location of nodal solutions for problem $%
\left( \mathrm{P}\right) $. We will posit the hypothesis below.

\begin{itemize}
\item[$(\mathrm{H}.3)$] There exist $\alpha ,\beta ,M>0$ such that 
\begin{equation*}
\max \{\alpha ,\beta \}<p-1
\end{equation*}%
and, moreover, 
\begin{equation*}
|f(x,s,\xi )|\leq M(1+|s|^{\alpha }+|\xi |^{\beta })\text{,}
\end{equation*}%
for all $(x,s,\xi )\in \Omega \times 
\mathbb{R}
\times \mathbb{R}^{N}$.
\end{itemize}

\begin{lemma}
\label{lemma-1}Assume $(\mathrm{H}.2)$ and $(\mathrm{H}.3)$ are fulfilled.

\begin{itemize}
\item[$\mathrm{(}i\mathrm{)}$] If $\overline{u}_{1},\overline{u}_{2}\in
W^{1,p}(\Omega )\cap L^{\infty }(\Omega )$ are supersolutions for problem $%
\left( \mathrm{P}\right) $, then $\overline{u}=\min \{\overline{u}_{1},%
\overline{u}_{2}\}$ is also a supersolution for problem $\left( \mathrm{P}%
\right) $.

\item[$\mathrm{(}ii\mathrm{)}$] If $\underline{u}_{1},\underline{u}_{2}\in
W^{1,p}(\Omega )\cap L^{\infty }(\Omega )$ are subsolutions for problem $%
\left( \mathrm{P}\right) $, then $\underline{u}=\max \{\underline{u}_{1},%
\underline{u}_{2}\}$ is also a subsolution for problem $\left( \mathrm{P}%
\right) $.
\end{itemize}
\end{lemma}

\begin{proof}
We provide the argument only for part $\mathrm{(}i\mathrm{)}$ because $%
\mathrm{(}ii\mathrm{)}$ can be similarly established. Inspired by \cite[%
Theorem 3.20]{CLM}, \cite[Lemma 1]{CLMous} and the proof of \cite[Lemma 3]%
{MMP}, for a fixed $\varepsilon >0,$ let us define the truncation function $%
\xi _{\varepsilon }(s)=\max \{-\varepsilon ,\min \{s,\varepsilon \}\}$ for $%
s\in 
\mathbb{R}
.$ It is shown in \cite{MM} that $\xi _{\varepsilon }((\overline{u}_{1}-%
\overline{u}_{2})^{-})\in W^{1,p}(\Omega ),$%
\begin{equation*}
\nabla \xi _{\varepsilon }((\overline{u}_{1}-\overline{u}_{2})^{-})=\xi
_{\varepsilon }^{\prime }((\overline{u}_{1}-\overline{u}_{2})^{-})\nabla (%
\overline{u}_{1}-\overline{u}_{2})^{-}
\end{equation*}%
For any test function $\varphi \in C_{c}^{1}(\Omega )$ with $\varphi \geq 0,$
it holds%
\begin{equation}
\begin{array}{l}
\left\langle -\Delta _{p}\overline{u}_{1}+\frac{|\nabla \overline{u}_{1}|^{p}%
}{\overline{u}_{1}+\delta },\xi _{\varepsilon }((\overline{u}_{1}-\overline{u%
}_{2})^{-})\varphi \right\rangle \\ 
\geq \int_{\Omega }f(x,\overline{u}_{1},\nabla \overline{u}_{1})\xi
_{\varepsilon }((\overline{u}_{1}-\overline{u}_{2})^{-})\varphi \text{ }%
\mathrm{d}x,%
\end{array}
\label{40}
\end{equation}%
and%
\begin{equation}
\begin{array}{l}
\left\langle -\Delta _{p}\overline{u}_{2}+\frac{|\nabla \overline{u}_{2}|^{p}%
}{\overline{u}_{2}+\delta },(\varepsilon -\xi _{\varepsilon }((\overline{u}%
_{1}-\overline{u}_{2})^{-}))\varphi \right\rangle \\ 
\geq \int_{\Omega }f(x,\overline{u}_{2},\nabla \overline{u}_{2})\left(
\varepsilon -\xi _{\varepsilon }((\overline{u}_{1}-\overline{u}%
_{2})^{-})\right) \varphi \text{ }\mathrm{d}x.%
\end{array}
\label{41}
\end{equation}%
On the other hand, using the monotonicity of the $p$-Laplacian operator, we
get 
\begin{equation}
\begin{array}{l}
\left\langle -\Delta _{p}\overline{u}_{1}+\frac{|\nabla \overline{u}_{1}|^{p}%
}{\overline{u}_{1}+\delta },\xi _{\varepsilon }((\overline{u}_{1}-\overline{u%
}_{2})^{-})\varphi \right\rangle \\ 
\text{ \ \ \ \ \ \ \ \ \ \ \ }+\left\langle -\Delta _{p}\overline{u}_{2}+%
\frac{|\nabla \overline{u}_{2}|^{p}}{\overline{u}_{2}+\delta },(\varepsilon
-\xi _{\varepsilon }((\overline{u}_{1}-\overline{u}_{2})^{-}))\varphi
\right\rangle \\ 
\leq \int_{\Omega }|\nabla \overline{u}_{1}|^{p-2}(\nabla \overline{u}%
_{1},\nabla \varphi )_{%
\mathbb{R}
^{N}}\xi _{\varepsilon }((\overline{u}_{1}-\overline{u}_{2})^{-})\text{ }%
\mathrm{d}x \\ 
\text{ \ \ \ \ \ \ \ \ \ \ \ }+\int_{\Omega }\frac{|\nabla \overline{u}%
_{1}|^{p}}{\overline{u}_{1}+\delta }\xi _{\varepsilon }((\overline{u}_{1}-%
\overline{u}_{2})^{-})\varphi \text{ }\mathrm{d}x \\ 
\text{ \ \ \ \ \ \ \ \ \ \ \ }+\int_{\Omega }|\nabla \overline{u}%
_{2}|^{p-2}(\nabla \overline{u}_{2},\nabla \varphi )_{%
\mathbb{R}
^{N}}\left( \varepsilon -\xi _{\varepsilon }((\overline{u}_{1}-\overline{u}%
_{2})^{-})\right) \text{ }\mathrm{d}x \\ 
\text{ \ \ \ \ \ \ \ \ \ \ \ }+\int_{\Omega }\frac{|\nabla \overline{u}%
_{2}|^{p}}{\overline{u}_{2}+\delta }(\varepsilon -\xi _{\varepsilon }((%
\overline{u}_{1}-\overline{u}_{2})^{-}))\varphi \text{ }\mathrm{d}x.%
\end{array}
\label{42}
\end{equation}%
Then, gathering (\ref{40}) together with (\ref{41}), by means of (\ref{42}),
one gets 
\begin{equation*}
\begin{array}{l}
\int_{\Omega }|\nabla \overline{u}_{1}|^{p-2}(\nabla \overline{u}_{1},\nabla
\varphi )_{%
\mathbb{R}
^{N}}\frac{1}{\varepsilon }\xi _{\varepsilon }((\overline{u}_{1}-\overline{u}%
_{2})^{-})\text{ }\mathrm{d}x \\ 
\text{ \ \ \ \ \ \ \ \ \ \ \ }+\int_{\Omega }\frac{|\nabla \overline{u}%
_{1}|^{p}}{\overline{u}_{1}+\delta }\frac{1}{\varepsilon }\xi _{\varepsilon
}((\overline{u}_{1}-\overline{u}_{2})^{-}\text{ }\mathrm{d}x \\ 
\text{ \ \ \ \ \ \ \ \ \ \ \ }+\int_{\Omega }|\nabla \overline{u}%
_{2}|^{p-2}(\nabla \overline{u}_{2},\nabla \varphi )_{%
\mathbb{R}
^{N}}\left( 1-\frac{1}{\varepsilon }\xi _{\varepsilon }((\overline{u}_{1}-%
\overline{u}_{2})^{-})\right) \text{ }\mathrm{d}x \\ 
\text{ \ \ \ \ \ \ \ \ \ \ \ }+\int_{\Omega }\frac{|\nabla \overline{u}%
_{2}|^{p}}{\overline{u}_{2}+\delta }(1-\frac{1}{\varepsilon }\xi
_{\varepsilon }((\overline{u}_{1}-\overline{u}_{2})^{-})\text{ }\mathrm{d}x
\\ 
\geq \int_{\Omega }f(x,\overline{u}_{1},\nabla \overline{u}_{1})\frac{1}{%
\varepsilon }\xi _{\varepsilon }((\overline{u}_{1}-\overline{u}%
_{2})^{-})\varphi \text{ }\mathrm{d}x \\ 
\text{ \ \ \ \ \ \ \ \ \ \ \ }+\int_{\Omega }f(x,\overline{u}_{2},\nabla 
\overline{u}_{2})\left( 1-\frac{1}{\varepsilon }\xi _{\varepsilon }((%
\overline{u}_{1}-\overline{u}_{2})^{-})\right) \varphi \text{ }\mathrm{d}x,%
\end{array}%
\end{equation*}%
Passing to the limit as $\varepsilon \rightarrow 0$ and noticing that%
\begin{equation*}
\frac{1}{\varepsilon }\xi _{\varepsilon }((\overline{u}_{1}-\overline{u}%
_{2})^{-}\rightarrow \mathbbm{1}_{\{\overline{u}_{1}<\overline{u}_{2}\}}(x)%
\text{, \ a.e. in }\Omega \text{ as }\varepsilon \rightarrow 0,
\end{equation*}%
we obtain%
\begin{equation*}
\int_{\Omega }|\nabla \overline{u}|^{p-2}\nabla \overline{u}\nabla \varphi 
\text{ }\mathrm{d}x+\int_{\Omega }\frac{|\nabla \overline{u}|^{p}}{\overline{%
u}+\delta }\varphi \text{ }\mathrm{d}x\geq \int_{\Omega }f(x,\overline{u}%
,\nabla \overline{u})\varphi \text{ }\mathrm{d}x,
\end{equation*}%
for all $\varphi \in C_{c}^{1}(\Omega ),$ $\varphi \geq 0$ a.e. in $\Omega $%
. Since $C_{c}^{1}(\Omega )$ is dense in $W^{1,p}(\Omega ),$ we achieve the
desired conclusion.
\end{proof}

Inspired by \cite{MMM}, next we set forth a result addressing location of
solutions and a priori estimates for problem $\left( \mathrm{P}\right) $.

\begin{theorem}
\label{T2} Assume that condition $(\mathrm{H}.2)-(\mathrm{H}.3)$ are
fulfilled.

\begin{itemize}
\item[\textrm{(}$i$\textrm{)}] If $f(x,0,0)\geq 0$ for a.e. $x\in \Omega $,
then for every nodal solution $u_{0}\in \lbrack \underline{u},\overline{u}]$
of problem $\left( \mathrm{P}\right) $ there exists a nontrivial solution $%
u_{+}$ of $\left( \mathrm{P}\right) $ such that $u_{0}\leq u_{+}\leq 
\overline{u}_{+}$ and $u_{+}\geq 0$ on $\Omega $.

\item[\textrm{(}$ii$\textrm{)}] If $f(x,0,0)\leq 0$ for a.e. $x\in \Omega $,
then for every nodal solution $u_{0}\in \lbrack \underline{u},\overline{u}]$
of problem $\left( \mathrm{P}\right) $ there exists a nontrivial solution $%
u_{-}$ of $\left( \mathrm{P}\right) $ such that $u_{0}\geq u_{-}\geq 
\overline{u}_{-}$ and $u_{-}\leq 0$ on $\Omega $.

\item[\textrm{(}$iii$\textrm{)}] If $f(x,0,0)=0$ for a.e. $x\in \Omega $,
then for every nodal solution $u_{0}\in \lbrack \underline{u},\overline{u}]$
of problem $\left( \mathrm{P}\right) $ there exist two other nontrivial
solutions $u_{+}$ and $u_{-}$ of $\left( \mathrm{P}\right) $ such that $%
u_{-}\leq u_{0}\leq u_{+}$, $u_{+}\geq 0$ and $u_{-}\leq 0$ on $\Omega $.
\end{itemize}
\end{theorem}

\begin{proof}
\textrm{(}$i$\textrm{)} Let $u_{0}$ be a nodal solution of problem $\left( 
\mathrm{P}\right) $ within $[\underline{u},\overline{u}]$. The assumption $%
f(x,0,0)\geq 0$ for a.e. $x\in \Omega $ ensures that $0$ is a subsolution of
problem $\left( \mathrm{P}\right) $. By Lemma \ref{lemma-1}, part $\mathrm{(}%
i\mathrm{),}$ we infer that $u_{0,+}:=\max \{0,u_{0}\}\ $is a subsolution of
problem $\left( \mathrm{P}\right) $ which obviously satisfies $u_{0,+}\leq 
\overline{u}_{+}$, with $\overline{u}_{+}:=\max \{0,\overline{u}\}$. So, in
view of Theorem \ref{T3}, there exists a solution $u_{+}\in \mathcal{C}^{1}(%
\overline{\Omega })$ of $\left( \mathrm{P}\right) $ within $[u_{0,+},%
\overline{u}_{+}]$. Since the solution $u_{0}$ of $\left( \mathrm{P}\right) $
is nodal, its positive part $u_{0,+}$ is strictly positive on a subset of $%
\Omega $ of positive measure. Hence, $u_{+}$ is positive.

\textrm{(}$ii$\textrm{)} Let $u_{0}$ be a nodal solution of problem $\left( 
\mathrm{P}\right) $ within $[\underline{u},\overline{u}]$. The assumption $%
f(x,0,0)\leq 0$ for a.e. $x\in \Omega $ ensures that $0$ is a supersolution
of problem $\left( \mathrm{P}\right) $. By Lemma \ref{lemma-1}, part $%
\mathrm{(}ii\mathrm{)}$, we infer that $u_{0,-}:=\min \{0,u_{0}\}\ $is a
supersolution of problem $\left( \mathrm{P}\right) $ which clearly satisfies 
$u_{0,-}\geq \overline{u}_{-}$, with $\overline{u}_{-}:=\min \{0,\overline{u}%
\}$. Then, Theorem \ref{T3} implies that there exits a solution $u_{-}\in 
\mathcal{C}^{1}(\overline{\Omega })$ of $\left( \mathrm{P}\right) $ within $[%
\overline{u}_{-},u_{0,-}]$. Recalling that the solution $u_{0}$ of $\left( 
\mathrm{P}\right) $ is nodal, its negative part $u_{0,-}$ is strictly
negative on a subset of $\Omega $ of positive measure. Therefore, $u_{-}\leq
0$ and $u_{-}\neq 0$.

\textrm{(}$iii$\textrm{) }If $f(x,0,0)=0$ for a.e. $x\in \Omega $, then the
assertions \textrm{(}$i$\textrm{)} and \textrm{(}$ii$\textrm{)} can be
applied simultaneously, giving rise to two nontrivial opposite constant-sign
solutions $u_{+}$ and $u_{-}$ of problem $\left( \mathrm{P}\right) $ with
the properties required in the statement.
\end{proof}

\section{Nodal solutions}

\label{S4}

In this section, beside $(\mathrm{H}_{1}),$ we will posit the hypothesis
below.

\begin{itemize}
\item[$(\mathrm{H.}4)$] With appropriate $m>0$ one has 
\begin{equation*}
\lim_{|s|\rightarrow 0}\inf \{f(x,s,\xi ):\xi \in 
\mathbb{R}
^{N}\}>m\text{,}
\end{equation*}%
uniformly in $x\in \Omega $.
\end{itemize}

Our first goal is to construct sub-and-supersolution pairs of $\left( 
\mathrm{P}\right) $. With this aim, consider the homogeneous Dirichlet
problem 
\begin{equation}
-\Delta _{p}z=1,\text{ in }\Omega ,\text{ \ }z=0\text{ on }\partial \Omega 
\text{,}  \label{5bis}
\end{equation}%
which admits a unique solution $z\in \mathcal{C}^{1,\tau }(\overline{\Omega }%
)$ satisfying 
\begin{equation}
\Vert z\Vert _{C^{1,\tau }(\overline{\Omega })}\leq L\text{,}  \label{3}
\end{equation}%
\begin{equation}
\frac{d(x)}{c}\leq z\leq cd(x)\text{\ in\ }\Omega \text{,\ \ \ \ }\frac{%
\partial z}{\partial \eta }<0\text{\ on\ }\partial \Omega \text{,}
\label{12*}
\end{equation}%
for certain constant $c>1$.

Now, given $0<\delta <\mathrm{diam}(\Omega )$, denote by $z_{\delta }\in 
\mathcal{C}^{1,\tau }(\overline{\Omega })$ the solution of the Dirichlet
problem 
\begin{equation}
-\Delta _{p}u=\left\{ 
\begin{array}{ll}
1 & \text{if }x\in \Omega \backslash \overline{\Omega }_{\delta }, \\ 
-1 & \text{otherwise},%
\end{array}%
\right. \quad u=0\text{ on }\partial \Omega ,  \label{1}
\end{equation}%
Existence and uniqueness directly stem from Minty-Browder's Theorem \cite{B}
while $\mathcal{C}^{1,\tau }(\overline{\Omega })$ regularity follows from
Lieberman's regularity Theorem \cite{L}. Moreover, the weak comparison
principle implies that 
\begin{equation}
z_{\delta }\leq z\ \text{in }\Omega ,  \label{3*}
\end{equation}%
while for $\delta >0$ small enough it holds%
\begin{equation}
\frac{\partial z_{\delta }}{\partial \eta }<\frac{1}{2}\frac{\partial z}{%
\partial \eta }<0\text{ on }\partial \Omega \text{ \ and \ }z_{\delta }\geq 
\frac{1}{2}\text{\thinspace }z\text{ \ in }\Omega  \label{4*}
\end{equation}%
(see \cite{H}).

Define 
\begin{equation}
\underline{u}:=\delta ^{\frac{1}{p}}z_{\delta }^{\omega }-\delta \text{, \ \
\ }\overline{u}:=\delta ^{-p}z^{\overline{\omega }}-\delta \text{,}
\label{32}
\end{equation}%
where 
\begin{equation}
\frac{\omega -1}{\omega }>\frac{1}{p-1}>\frac{\overline{\omega }-1}{%
\overline{\omega }}\text{ \ with }\omega >\overline{\omega }>1  \label{4}
\end{equation}%
and%
\begin{equation}
\overline{\omega }<1+p(1-\frac{\max \{\alpha ,\beta \}}{p-1}).  \label{5}
\end{equation}%
From (\ref{32}), (\ref{3}), (\ref{3*}) and (\ref{12*}), it follows 
\begin{equation}
\overline{u}\leq \delta ^{-p}(Ld)^{\overline{\omega }}\text{ \ and \ }\Vert
\nabla \overline{u}\Vert _{\infty }\leq \delta ^{-p}\hat{L},  \label{2}
\end{equation}%
with $\hat{L}:=\overline{\omega }L^{\overline{\omega }}$. Moreover,%
\begin{equation}
\left\{ 
\begin{array}{l}
\frac{\partial \underline{u}}{\partial \eta }=\delta ^{\frac{1}{p}}\frac{%
\partial (z_{\delta }^{\omega })}{\partial \eta }=\delta ^{\frac{1}{p}%
}\omega z_{\delta }^{\omega -1}\frac{\partial z_{\delta }}{\partial \eta }=0
\\ 
\frac{\partial \overline{u}}{\partial \eta }=\delta ^{-p}\frac{\partial (z^{%
\overline{\omega }})}{\partial \eta }=\delta ^{-p}\overline{\omega }z^{%
\overline{\omega }-1}\frac{\partial z}{\partial \eta }=0%
\end{array}%
\right. \text{ on }\partial \Omega ,  \label{9}
\end{equation}%
because $z,$ $z_{\delta }$ solve (\ref{5bis}), (\ref{1}), respectively and $%
\omega ,\overline{\omega }>1$.

We claim that $\underline{u}\leq \overline{u}$ with a small $\delta >0$.
Indeed, since $\omega >\overline{\omega }$ and $z_{\delta }\leq z$ for all $%
\delta <\mathrm{diam}(\Omega ),$ it follows that%
\begin{equation*}
\begin{array}{l}
\overline{u}(x)-\underline{u}(x)=\left( \delta ^{-p}z^{\bar{\omega}}-\delta
\right) -\left( \delta ^{\frac{1}{p}}z_{\delta }^{\omega }-\delta \right) \\ 
\geq \delta ^{-p}z^{\bar{\omega}}-\delta ^{\frac{1}{p}}z^{\omega }=z^{\omega
}(\delta ^{-p}z^{\bar{\omega}-\omega }-\delta ^{\frac{1}{p}}) \\ 
\geq z^{\omega }(\delta ^{-p}(cd(x))^{\bar{\omega}-\omega }-\delta ^{\frac{1%
}{p}})\geq 0,%
\end{array}%
\end{equation*}%
provided that $\delta >0$ is small. Thus $\underline{u}\leq \overline{u}$ in 
$\Omega $, as desired.

\begin{theorem}
\label{T1} Let $(\mathrm{H.}3)$--$(\mathrm{H.}4)$ be satisfied. Then problem 
$\left( \mathrm{P}\right) $ admits a nodal solution $u_{0}\in \mathcal{C}%
^{1}(\overline{\Omega })$ such that $u_{0}(x)$ is negative once $%
d(x)\rightarrow 0$. Furthermore, there exists a positive solution $u_{+}\in 
\mathcal{C}^{1}(\overline{\Omega })$ of $\left( \mathrm{P}\right) $ with $%
u_{+}(x)$ is zero once $d(x)\rightarrow 0$.
\end{theorem}

\begin{proof}
Let us show that the function $\overline{u}$ given by (\ref{32}) satisfies (%
\ref{c2}). With this aim, pick $u\in W^{1,p}(\Omega )$ such that $-\overline{%
u}\leq u\leq \overline{u}$. From $(\mathrm{H.}3),$ (\ref{2}), it follows%
\begin{equation}
\begin{split}
|f(\cdot ,\overline{u},\nabla \overline{u})|& \leq M(1+|\overline{u}%
|^{\alpha }+|\nabla \overline{u}|^{\beta }) \\
& \leq M(1+(\delta ^{-p}(Ld)^{\overline{\omega }})^{\alpha }+(\delta ^{-p}%
\hat{L})^{\beta } \\
& \leq C\delta ^{-p\max \{\alpha ,\beta \}}\text{,}
\end{split}
\label{6}
\end{equation}%
for some constant $C>0$ and for $\delta >0$ sufficiently small. A direct
computation gives 
\begin{equation*}
\begin{array}{l}
-\Delta _{p}z^{\bar{\omega}}+\lambda \frac{|\nabla z^{\bar{\omega}}|^{p}}{z^{%
\bar{\omega}}}=\bar{\omega}^{p-1}\left( 1-(\bar{\omega}-1)\left( p-1\right) 
\frac{|\nabla z|^{p}}{z}\right) z^{(\bar{\omega}-1)(p-1)}+\bar{\omega}^{p}%
\frac{z^{(\bar{\omega}-1)p}|\nabla z|^{p}}{z^{\bar{\omega}}} \\ 
=\bar{\omega}^{p-1}\left( 1-(\bar{\omega}-1)\left( p-1\right) \frac{|\nabla
z|^{p}}{z}\right) z^{(\bar{\omega}-1)(p-1)}+\bar{\omega}^{p}z^{(\bar{\omega}%
-1)(p-1)}\frac{z^{\bar{\omega}-1}|\nabla z|^{p}}{z^{\bar{\omega}}} \\ 
=\bar{\omega}^{p-1}\left[ 1+\bar{\omega}(1-\frac{(\bar{\omega}-1)\left(
p-1\right) }{\bar{\omega}})\frac{|\nabla z|^{p}}{z}\right] z^{(\bar{\omega}%
-1)(p-1)}.%
\end{array}%
\end{equation*}%
Using (\ref{32}) and (\ref{4}) one has%
\begin{equation}
\begin{array}{l}
-\Delta _{p}\overline{u}+\frac{|\nabla \overline{u}|^{p}}{\overline{u}%
+\delta }=\delta ^{-p(p-1)}(-\Delta _{p}z^{\bar{\omega}}+\frac{|\nabla z^{%
\bar{\omega}}|^{p}}{z^{\bar{\omega}}}) \\ 
=\delta ^{-p(p-1)}\bar{\omega}^{p-1}\left[ 1+\bar{\omega}(1-\frac{(\bar{%
\omega}-1)\left( p-1\right) }{\bar{\omega}})\frac{|\nabla z|^{p}}{z}\right]
z^{(\bar{\omega}-1)(p-1)} \\ 
\geq \delta ^{-p(p-1)}\bar{\omega}^{p-1}\left\{ 
\begin{array}{ll}
z^{(\bar{\omega}-1)(p-1)} & \text{in}\;\Omega \backslash \overline{\Omega }%
_{\delta }, \\ 
\bar{\omega}(1-\frac{(\bar{\omega}-1)\left( p-1\right) }{\bar{\omega}})z^{(%
\bar{\omega}-1)(p-1)-1}|\nabla z|^{p} & \text{in}\;\Omega _{\delta }\text{.}%
\end{array}%
\right.%
\end{array}
\label{7}
\end{equation}%
Thus, after decreasing $\delta $ if necessary, we achieve%
\begin{equation}
\begin{array}{l}
\delta ^{-p(p-1)}\bar{\omega}^{p-1}z^{(\bar{\omega}-1)(p-1)} \\ 
\geq \delta ^{-p(p-1)}\bar{\omega}^{p-1}(c^{-1}d(x))^{(\bar{\omega}-1)(p-1)}
\\ 
\geq \delta ^{-p(p-1)}\bar{\omega}^{p-1}(c^{-1}\delta )^{(\bar{\omega}%
-1)(p-1)} \\ 
=\delta ^{(\bar{\omega}-1-p)(p-1)}\bar{\omega}^{p-1}c^{-(\bar{\omega}%
-1)(p-1)} \\ 
\geq \delta ^{-p\max \{\alpha ,\beta \}}\text{ \ in }\Omega \backslash 
\overline{\Omega }_{\delta },%
\end{array}
\label{8}
\end{equation}%
because of (\ref{5}).\ Thus, (\ref{6})--(\ref{8}) yield 
\begin{equation*}
-\Delta _{p}\overline{u}+\frac{|\nabla \overline{u}|^{p}}{\overline{u}%
+\delta }\geq f(\cdot ,\overline{u},\nabla \overline{u})\text{\ \ in\ }%
\Omega \backslash \overline{\Omega }_{\delta }\text{.}
\end{equation*}%
Let now $x\in {\Omega }_{\delta }$. From (\ref{12*}) and (\ref{4}), one can
find a constant $\bar{\mu}>0$ such that 
\begin{equation*}
(1-\frac{(\bar{\omega}-1)\left( p-1\right) }{\bar{\omega}})|\nabla z|>\bar{%
\mu}\text{ \ in }{\Omega }_{\delta }.
\end{equation*}%
Then, (\ref{12*}) and (\ref{5}) entail%
\begin{equation*}
\begin{array}{l}
\delta ^{-p(p-1)}\bar{\omega}^{p}(1-\frac{(\bar{\omega}-1)\left( p-1\right) 
}{\bar{\omega}})z^{(\bar{\omega}-1)(p-1)-1}|\nabla z{|}^{p} \\ 
\geq \delta ^{-p(p-1)}\bar{\omega}^{p}(cd(x))^{(\bar{\omega}-1)(p-1)-1}\bar{%
\mu}^{p} \\ 
\geq \delta ^{-p(p-1)}\bar{\omega}^{p}(c\delta )^{(\bar{\omega}-1)(p-1)-1}%
\bar{\mu}^{p} \\ 
\geq \delta ^{-p\max \{\alpha ,\beta \}}\text{ \ in }\Omega _{\delta }\text{,%
}%
\end{array}%
\end{equation*}%
for $\delta >0$\ small enough, that is 
\begin{equation*}
-\Delta _{p}\overline{u}+\frac{|\nabla \overline{u}|^{p}}{\overline{u}%
+\delta }\geq f(\cdot ,\overline{u},\nabla \overline{u})\text{\ \ in\ }%
\Omega _{\delta }\text{.}
\end{equation*}%
Summing up, 
\begin{equation*}
-\Delta _{p}\overline{u}+\frac{|\nabla \overline{u}|^{p}}{\overline{u}%
+\delta }\geq f(\cdot ,\overline{u},\nabla \overline{u})\text{\ on the
whole\ }\Omega \text{.}
\end{equation*}%
Finally, test with $\varphi \in W_{b}^{1,p}(\Omega ),$ $\varphi \geq 0$ a.e.
in $\Omega $, and recall (\ref{9}) to get%
\begin{equation*}
\begin{array}{l}
\int_{\Omega }|\nabla \overline{u}|^{p-2}\nabla \overline{u}\nabla \varphi \,%
\mathrm{d}x+\int_{\Omega }\frac{|\nabla \overline{u}|^{p}}{\overline{u}%
+\delta }\varphi \mathrm{d}x-\left\langle \frac{\partial \overline{u}}{%
\partial \eta _{p}},\gamma _{0}(\varphi )\right\rangle _{\partial \Omega }
\\ 
\geq \int_{\Omega }f(\cdot ,\overline{u},\nabla \overline{u})\varphi \text{%
\thinspace \textrm{d}}x\text{,}%
\end{array}%
\end{equation*}%
as desired. Here, $\gamma _{0}$ is the trace operator on $\partial \Omega $, 
\begin{equation}
\frac{\partial w}{\partial \eta _{p}}:=|\nabla w|^{p-2}\frac{\partial w}{%
\partial \eta },\quad \forall \,w\in W^{1,p}(\Omega )\cap C^{1}(\overline{%
\Omega }),  \label{conode}
\end{equation}%
while $\left\langle \cdot ,\cdot \right\rangle _{\partial \Omega }$ denotes
the duality brackets for the pair 
\begin{equation*}
(W^{1/p^{\prime },p}(\partial \Omega ),W^{-1/p^{\prime },p^{\prime
}}(\partial \Omega )).
\end{equation*}%
Next, we show that the function $\underline{u}$ in (\ref{32}) satisfies (\ref%
{c2}). In $\Omega \backslash \overline{\Omega }_{\delta }$, a direct
computation gives 
\begin{equation*}
\begin{array}{l}
-\Delta _{p}z_{\delta }^{\omega }+\frac{|\nabla z_{\delta }^{\omega }|^{p}}{%
z_{\delta }^{\omega }}=\omega ^{p-1}\left( 1-(\omega -1)\left( p-1\right) 
\frac{|\nabla z_{\delta }|^{p}}{z_{\delta }}\right) z_{\delta }^{(\omega
-1)(p-1)}+\omega ^{p}\frac{z_{\delta }^{(\omega -1)p}|\nabla z_{\delta }|^{p}%
}{z_{\delta }^{\omega }} \\ 
=\omega ^{p-1}\left[ 1+\omega (1-\frac{(\omega -1)\left( p-1\right) }{\omega 
})\frac{|\nabla z_{\delta }|^{p}}{z_{\delta }}\right] z_{\delta }^{(\omega
-1)(p-1)},%
\end{array}%
\end{equation*}%
while in $\Omega _{\delta },$ we get%
\begin{equation*}
\begin{array}{l}
-\Delta _{p}z_{\delta }^{\omega }+\frac{|\nabla z_{\delta }^{\omega }|^{p}}{%
z_{\delta }^{\omega }}=\omega ^{p-1}\left[ -1+\omega (1-\frac{(\omega
-1)\left( p-1\right) }{\omega })\frac{|\nabla z_{\delta }|^{p}}{z_{\delta }}%
\right] z_{\delta }^{(\omega -1)(p-1)}\text{.}%
\end{array}%
\end{equation*}%
Thus, by (\ref{32}) and due to (\ref{4}), we have 
\begin{equation*}
-\Delta _{p}\underline{u}+\frac{|\nabla \underline{u}|^{p}}{\underline{u}%
+\delta }=\delta ^{\frac{1}{p^{\prime }}}(-\Delta _{p}z_{\delta }^{\omega }+%
\frac{|\nabla z_{\delta }^{\omega }|^{p}}{z_{\delta }^{\omega }})\leq
\left\{ 
\begin{array}{ll}
\delta ^{\frac{1}{p^{\prime }}}\omega ^{p-1}z_{\delta }^{(\omega -1)(p-1)} & 
\text{in }\Omega \backslash \overline{\Omega }_{\delta } \\ 
0 & \text{in }\Omega _{\delta }.%
\end{array}%
\right.
\end{equation*}%
Hence, on account of (\ref{3*}), (\ref{3}) and for an appropriate constant $%
m $ in $(\mathrm{H.}4)$ chosen so that 
\begin{equation*}
m>\delta ^{\frac{1}{p^{\prime }}}\omega ^{p-1}L^{(\omega -1)(p-1)}\text{ \
for }\delta >0\text{ sufficiently small,}
\end{equation*}%
we get 
\begin{equation}
-\Delta _{p}\underline{u}+\frac{|\nabla \underline{u}|^{p}}{\underline{u}%
+\delta }\leq f(\cdot ,\underline{u},\nabla \underline{u}).  \label{46}
\end{equation}%
Finally, test (\ref{46}) with $\varphi \in W_{b}^{1,p}(\Omega ),$ $\varphi
\geq 0$ a.e. in $\Omega ,$ and recall (\ref{9}), besides Green's formula 
\cite{CF}, to arrive at%
\begin{equation*}
\begin{array}{l}
\int_{\Omega }|\nabla \underline{u}|^{p-2}\nabla \underline{u}\nabla \varphi
\,\mathrm{d}x+\int_{\Omega }\frac{|\nabla \underline{u}|^{p}}{\underline{u}%
+\delta }\varphi \,\mathrm{d}x \\ 
\leq \int_{\Omega }|\nabla \underline{u}|^{p-2}\nabla \underline{u}\nabla
\varphi \,\mathrm{d}x-\left\langle \frac{\partial \underline{u}}{\partial
\eta _{p}},\gamma _{0}(\varphi )\right\rangle _{\partial \Omega
}+\int_{\Omega }\frac{|\nabla \underline{u}|^{p}}{\underline{u}+\delta }%
\varphi \,\mathrm{d}x \\ 
=\int_{\Omega }(-\Delta _{p}\underline{u}\ +\frac{|\nabla \underline{u}|^{p}%
}{\underline{u}+\delta })\varphi \,\mathrm{d}x\leq \int_{\Omega }f(\cdot ,%
\underline{u},\nabla \underline{u})\varphi \,\mathrm{d}x,%
\end{array}%
\end{equation*}%
because $\gamma _{0}(\varphi )\geq 0$ whatever $\varphi \in W^{1,p}(\Omega )$%
, $\varphi \geq 0$ a.e. in $\Omega $ (see \cite[p. 35]{CLM}).

Therefore, $\underline{u}$ and $\overline{u}$ satisfy assumption $(\mathrm{H.%
}2)$, whence Theorem \ref{T2} can be applied, and there exists a solution $%
u_{0}\in \mathcal{C}^{1,\tau }(\overline{\Omega }),$ $\tau \in ]0,1[,$ of
problem $\left( \mathrm{P}\right) $ such that 
\begin{equation}
\underline{u}\leq u_{0}\leq \overline{u}.  \label{11}
\end{equation}%
Moreover, $u_{0}$ is nodal. In fact, through (\ref{32}) and (\ref{12*}) we
obtain%
\begin{equation*}
\overline{u}=\delta ^{-p}z^{\overline{\omega }}-\delta \leq \delta
^{-p}(cd(x))^{\overline{\omega }}-\delta ,
\end{equation*}%
which actually means 
\begin{equation}
\overline{u}(x)<0\;\;\text{provided}\;\;d(x)<c^{-1}\delta ^{\frac{p+1}{%
\overline{\omega }}}.
\end{equation}%
Again, by (\ref{32})-(\ref{12*}) together with (\ref{4*}), it follows that%
\begin{equation*}
\underline{u}=\delta ^{\frac{1}{p}}z_{\delta }^{\omega }-\delta \geq \delta
^{\frac{1}{p}}(\frac{d(x)}{2c})^{\omega }-\delta ,
\end{equation*}%
which implies that%
\begin{equation}
\underline{u}(x)>0\;\;\text{as soon as}\;\;d(x)>2c\delta ^{\frac{1}{\omega
p^{\prime }}}\text{.}  \label{12}
\end{equation}%
On account of (\ref{11})-(\ref{12}), the conclusion follows.

On the other hand, on the basis of $(\mathrm{H.}4)$ and bearing in mind
Theorem \ref{T2}, there exists a nontrivial solution $u_{+}$ of $\left( 
\mathrm{P}\right) $ such that $u_{0}\leq u_{+}\leq \overline{u}_{+}$ and $%
u_{+}\geq 0$ on $\Omega $. In view of (\ref{32}), $\overline{u}_{+}=0$ once $%
d(x)\rightarrow 0$ and so $u_{+}$ vanishes as $d(x)\rightarrow 0$. This
completes the proof.
\end{proof}

\end{document}